\documentclass[11pt]{amsart}

\usepackage{fullpage, amsmath, amsthm, amssymb, color, caption, placeins}

\newcommand{\N}{ \mathbb N}
\newcommand{\Z}{ \mathbb Z}

\newcommand{\C}{ \mathbb C}
\newcommand{\R}{\mathbb{R}}
\newcommand{\FF}{\mathbb{F}}
\newcommand{\PP}{\mathbb{P}}
\newcommand{\G}{ \Gamma}
\newcommand{\lam}{\lambda}
\renewcommand{\l}{\lambda}
\newcommand{\PDn}{P_D^{(n)}}
\newcommand{\PPDn}{\mathbb{P}_D^{(n)}}
\newcommand{\FDn}{F_D^{(n)}}
\newcommand{\FFDn}{\mathbb{F}_D^{(n)}}

\newcommand{\PPDtwo}{\PP_D^{(2)}}
\newcommand{\FDtwo}{F_D^{(2)}}

\newcommand{\Ftwo}{{}_{2}F_{1}}

\newcommand{\PPtwo}{{}_{2}\mathbb{P}_{1}}
\newcommand{\PPn}{{}_{n+1}\mathbb{P}_{n}}
\newcommand{\FFtwo}{{}_{2}\mathbb{F}_{1}}
\newcommand{\FFn}{{}_{n+1}\mathbb{F}_{n}}

\renewcommand{\(}{\left(}
\renewcommand{\)}{\right)}
\newcommand{\ol}{\overline}
\newcommand{\eps}{\varepsilon}
\newcommand{\CC}{\binom}
\newcommand{\poch}[2]{(#1)_{#2}}

\newcommand{\FFqhat}{\widehat{\FF_q^\times}}

\newcommand{\dsp}{\displaystyle}

\newtheorem{theorem}{Theorem}[section]

\newtheorem{corollary}[theorem]{Corollary}
\newtheorem{definition}[theorem]{Definition}

\newtheorem{lemma}[theorem]{Lemma}

\newtheorem{proposition}[theorem]{Proposition}

\newcommand{\rmhs}{\color{black}}

\newcommand{\ch}[2]{\begin{pmatrix} #1 \\ #2
    \end{pmatrix}}

\renewcommand{\Re}{\mathrm{Re}}

\newcommand{\ClamNijk}{C_{\mathbf{\lam}}^{[N; i, j, \mathbf{k}]}}
\newcommand{\XlamNijk}{X_{\mathbf{\lam}}^{[N; i, j, \mathbf{k}]}}

\newcommand*\HYPERskip{&}
\catcode`,\active
\newcommand*\pFq{
\begingroup
\catcode`\,\active
\def ,{\HYPERskip}%
\doHyper
}
\catcode`\,12
\def\doHyper#1#2#3#4#5{%
\, _{#1}F_{#2}\left[\begin{matrix}#3 \smallskip \\  #4\end{matrix} \; ; \; #5\right]%
\endgroup
}

\newcommand*\HYPER{&}
\catcode`,\active
\newcommand*\pFFq{
\begingroup
\catcode`\,\active
\def ,{\HYPER}%
\doHyperF
}
\catcode`\,12
\def\doHyperF#1#2#3#4#5{%
\, _{#1}{\mathbb F}_{#2}\left[\begin{matrix}#3 \smallskip \\  #4\end{matrix} \; ; \; #5\right]%
\endgroup
}

\catcode`,\active

\catcode`\,12

\newcommand*\HYPERpp{&}
\catcode`,\active
\newcommand*\pPPq{
\begingroup
\catcode`\,\active
\def ,{\HYPERpp}%
\doHyperFpp
}
\catcode`\,12
\def\doHyperFpp#1#2#3#4#5{%
\, _{#1}{\mathbb P}_{#2}\left[\begin{matrix}#3 \smallskip \\  #4\end{matrix} \; ; \; #5\right]%
\endgroup
}

\newcommand*\Lauricella{&}
\catcode`,\active
\newcommand*\Ln{
\begingroup
\catcode`\,\active
\def ,{\Lauricella}%
\doLauricella
}
\catcode`\,12
\def\doLauricella#1#2#3#4#5{%
\, F^{(#1)}_D\left[\begin{matrix}#2; & #3 \smallskip \\  {} & #4 \end{matrix} \; ; \; #5 \right]%
\endgroup
}

\newcommand*\ApellFF{&}
\catcode`,\active
\newcommand*\FAFn{
\begingroup
\catcode`\,\active
\def ,{\ApellFF}%
\doApellFFp
}
\catcode`\,12
\def\doApellFFp#1#2#3#4#5{%
\, {\mathbb F}^{(#1)}_D\left[\begin{matrix}#2;&#3 \smallskip \\  {} & #4 \end{matrix} \; ; \;#5\right]%
\endgroup
}

\newcommand*\LauricellaP{&}
\catcode`,\active
\newcommand*\LnP{
\begingroup
\catcode`\,\active
\def ,{\LauricellaP}%
\doLauricellaP
}
\catcode`\,12
\def\doLauricellaP#1#2#3#4#5{%
\, P^{(#1)}_D\left[\begin{matrix}#2; & #3 \smallskip \\  {} & #4 \end{matrix} \; ; \; #5 \right]%
\endgroup
}

\newcommand*\ApellFp{&}
\catcode`,\active
\newcommand*\FAPn{
\begingroup
\catcode`\,\active
\def ,{\ApellFp}%
\doApellFpp
}
\catcode`\,12
\def\doApellFpp#1#2#3#4#5{%
\, {\mathbb P}^{(#1)}_D\left[\begin{matrix}#2;&#3 \smallskip \\  {} & #4 \end{matrix} \; ; \; #5\right]%
\endgroup
}

\title{A Cubic Transformation Formula for \\Appell-Lauricella hypergeometric functions \\over finite fields}

\author{Sharon Frechette}
\address{Department of Mathematics and Computer Science\\
  College of the Holy Cross\\
  Worcester, MA 01610} 
\email{sfrechet@mathcs.holycross.edu}
\author{Holly Swisher}
\address{Department of Mathematics\\
Oregon State University\\
Corvallis, OR 97331} \email{swisherh@math.oregonstate.edu}
\author{Fang-Ting Tu}
\address{Department of Mathematics\\
  Louisiana State University\\
  Baton Rouge, LA 70803} \email{ tu@math.lsu.edu}
\begin{document}


\begin{abstract}

We define a finite-field version of Appell-Lauricella hypergeometric functions built from period functions in several variables, paralleling the development by Fuselier, et. al \cite{dictionary} in the single variable case.  We develop geometric connections between these functions and the family of generalized Picard curves.  In our main result, we use finite-field Appell-Lauricella functions to establish a finite-field analogue of Koike and Shiga's cubic transformation \cite{koike-shiga1} for the Appell hypergeometric function $F_1$, proving a conjecture of Ling Long.  We use our multivariable period functions to construct formulas for the number of $\FF_p$-points on the generalized Picard curves.  We also give some transformation and reduction formulas for the period functions, and consequently for the finite-field Appell-Lauricella  functions.
\end{abstract}

\maketitle


\section{Introduction and Statement of Results}\label{section:intro}
Classical hypergeometric functions are among the most versatile of all special functions.  These functions and their finite-field analogues have numerous applications in number theory and geometry.  For instance, finite-field hypergeometric functions play a role in proving congruences and supercongruences, they count points modulo $p$ over algebraic varieties and affine hypersurfaces, and in certain instances they provide formulas for the Fourier coefficients of modular forms.

In this paper, we define functions $\FF_D^{(n)}$ as a finite-field version of the Appell-Lauricella hypergeometric functions $F_D^{(n)}$.  We develop the theory of these functions, with a focus on their geometric connections to the generalized Picard curves.  This parallels the construction (by the second and third authors, et. al.) in \cite{dictionary}, which was an effort to unify and improve on the interplay between classical and finite-field hypergeometric functions in the single-variable setting.

Our results are motivated by a conjecture of Ling Long, related to results by Koike and Shiga.  In \cite{koike-shiga1}, Koike and Shiga applied Appell's $F_1$-hypergeometric function in two variables to establish a new three-term arithmetic geometric mean result (AGM), related to Picard modular forms.   As a consequence of this cubic AGM, 
%
%
Koike and Shiga proved the following cubic transformation for Appell's $F_1$-function.
Let $x,y \in \mathbb{C}$, and let $\omega$ be a primitive cubic root of unity.  Then
\begin{multline}\label{F1-cubic}
F_1\left[ \frac{1}{3};\, \frac{1}{3},\,  \frac{1}{3};\, 1 \, \Big|\, 1-x^3,\,1-y^3\right]  \\
=
\frac{3}{1+x+y} \
F_1\left[ \frac{1}{3};\, \frac{1}{3},\, \frac{1}{3};\,1 \, \Big|\,
	\left(\frac{1+\omega x + \omega^2 y}{1 + x + y}\right)^3,\,
	\left(\frac{1+\omega^2 x + \omega y}{1 + x + y}\right)^3\, \right].
\end{multline}
%
%
As an application of Appell-Lauricella functions over finite fields, we prove the following finite-field analogue of Koike and Shiga's transformation, as conjectured by Ling Long.		
\begin{theorem}\label{finiteF1-cubic} 
Let $p\equiv 1 \pmod{3}$ be prime, let $\omega$ be a primitive cubic root of unity, and let $\eta_3$ be a primitive cubic character in $\widehat{\mathbb{F}_p^\times}$.  If $\lam,\mu \in \FF_p$ satisfy $1 + \lam + \mu \neq 0$,  then
\begin{align*}
\FAFn{2}{\eta_3}{\eta_3 &  \eta_3}{\eps}{1-\lam^3, 1-\mu^3} 
=
\FAFn{2}{\eta_3}{\eta_3 &  \eta_3}{\eps}
	{\left(\frac{1+\omega \lam + \omega^2 \mu}{1 + \lam + \mu}\right)^3, \left(\frac{1+\omega^2 \lam + \omega \mu}{1 + \lam + \mu}\right)^3\, }. 
\end{align*}
\end{theorem}

When $\lambda = \mu$, we have the following corollary. 
%
%
\begin{corollary}\label{finite2F1-cubic}
For $p\equiv 1 \pmod{3}$ prime, and $\omega$ as above, if $\lam \in \FF_p$ satisfies $1 + 2\lam \neq 0$,  then
\begin{align*}
\pFFq{2}{1}{\eta_3 & \eta_3^2}{&\eps}{1-\lam^3} 
=
\pFFq{2}{1}{\eta_3 & \eta_3^2}{&\eps}{\left(\frac{1- \lam}{1 + 2\lam}\right)^3}.
\end{align*}
\end{corollary}
The result of Corollary \ref{finite2F1-cubic} was first established in \cite{dictionary}, using a different method of proof.
It is a finite-field version of the cubic transformation 
%
%
\begin{equation}\label{2F1-cubic}
\pFq{2}{1}{\frac{1}{3},  \frac{2}{3}}{,1}{1-x^3} 
=
\frac{3}{1+2x} \
\pFq{2}{1}{\frac{1}{3},  \frac{2}{3}}{,1}{\left(\frac{1-x}{1+2x}\right)^3},
\end{equation}
proved by Borwein and Borwein \cite{borweins-1}, \cite{borweins-2} for $x \in \R$ with $0 < x < 1$, as a cubic analogue of Gauss' quadratic AGM.  

Taking the approach used in \cite{dictionary}, our finite-field Appell-Lauricella hypergeometric functions are defined in terms of finite-field period functions $\PP_D^{(n)}$.  These period functions are naturally related to periods of the generalized Picard curves 
\begin{equation}\label{C-lam}
\ClamNijk:\,\, y^N = x^i (1-x)^j (1-\lam_1 x)^{k_1} \cdots (1-\lam_n x)^{k_n}, 
\end{equation}
defined for distinct complex numbers $\lam_1, \ldots, \lam_n \neq 0,1$ and positive integers $N, i, j, k_1, \ldots, k_n$ with $\gcd(N, i, j, k_1, \ldots, k_n)=1$ and $N\nmid i+j+k_1+\cdots + k_n$.  As a consequence, the $\PP_D^{(n)}$ functions  are ideally suited for counting $\FF_p$-points on Picard curves; in Theorem \ref{Picard-pts}, we express these point counts in terms of finite sums of $\PP_D^{(n)}$ values.

Transformation and reduction formulas for classical hypergeometric functions have been successfully translated to the finite-field setting, first by Greene and also by authors such as McCarthy, and Fuselier et. al. (See \cite{greene},  \cite{mccarthy}, \cite{dictionary} for details.)  Transformation formulas for classical Appell-Lauricella hypergeometric functions can also be translated into the finite-field setting, using the same methods.  We carry out this process, proving several identities for the $\PP_D^{(n)}$- and $\FF_D^{(n)}$-functions.
We note that another version of finite-field Appell-Lauricella functions is independently defined by He \cite{he} and Li, et. al. \cite{li}, which closely follows Greene's definition.  For their version, they establish several degree 1 transformation and reduction formulas, including some that are analogous to the identities in this paper.
	
The paper is organized as follows:  Section \ref{section:prelim} includes definitions and properties needed for working with classical and finite-field Appell-Lauricella hypergeometric functions, such as Gamma and Jacobi functions, characters over finite fields, and classical hypergeometric functions of one variable.  In Section \ref{section:classical}, we recall the definitions of classical Appell-Lauricella functions $F_D^{(n)}$ and some of their transformation properties, and we define period functions $P_D^{(n)}$ to capture the relationship between Appell-Lauricella functions and Picard curves.  Section \ref{section:ff}  introduces the finite-field period functions $\PP_D^{(n)}$ and finite-field Appell-Lauricella functions $\FF_D^{(n)}$, as well as several of their transformation and reduction formulas.  In Section \ref{section:geometric}, we prove Theorem \ref{Picard-pts} which gives the number of $\FF_p$-points on the generalized Picard curves.  We also compute the genus of the generalized Picard curves $\ClamNijk$ in this section.  In Section \ref{section:trans}, we give several degree 1 transformation and reduction formulas for the $\PP_D^{(n)}$ and $\FF_D^{(n)}$ functions. Finally, in Section \ref{section:cubic}, we give the proofs of Theorem \ref{finiteF1-cubic} and Corollary \ref{finite2F1-cubic}.


\section{Preliminaries}\label{section:prelim}

In this section, we recall the necessary background for defining and working with classical and finite-field hypergeometric functions.  For further details, we refer the reader to \cite{AAR}, \cite{slater}, for the classical setting, and \cite{greene}, \cite{dictionary} for the finite-field setting.

%
%
For $a \in \C$ and $n \in \mathbb{Z}_{\geq 0}$, define the {\em Pochhammer symbol} $(a)_n$ by
$$
(a)_n:= \begin{cases} 1 & \text{ if $n=0$,}\\
a(a+1)\cdots(a+n-1) & \text{ if $n \geq 1$,}
\end{cases}
$$
and for $x \in \C$ with $\Re(x) > 0$, define the Gamma function $\Gamma(x)$ by 
$$
\Gamma(x) = \int_0^\infty t^{x-1} e^{-t}\, dt.
$$
Note that $(a)_n = \dsp \frac{\Gamma(a+n)}{\Gamma(a)}$ for all nonnegative integers $n$.  

%
%

For $r \geq 0$,  and parameters $a_1, a_2, \ldots, a_{r+1}$, \  $b_1, b_2, \ldots, b_r \in \C$ and $|x| < 1$, the {\em classical hypergeometric function} ${}_{r+1}F_r$ in these parameters is defined by

\begin{equation}\label{n+1Fn}
\pFq{r+1}{r}{a_1, a_2, \cdots, a_{r+1}}{,b_1, \cdots, b_r}{z}
:= \sum_{n=0}^\infty 
\frac{(a_1)_n (a_2)_n \cdots (a_{r+1})_n}{(b_1)_n (b_2)_n \cdots (b_r)_n\, n!} \: x^n.
\end{equation}
The study of these versatile functions goes back to the likes of Euler and Gauss, and among other things, the ${}_{r+1}F_r$-hypergeometric functions occur as solutions to hypergeometric differential equations and periods of algebraic varieties, among other things.

Finite-field hypergeometric functions in one variable were defined by Greene \cite{greene} as analogues to the classical version.
Turning our attention to this setting, let $\FF_q$ be a finite field where $q = p^e$, with $p$ an odd prime and $e \in \N$. Let $\FFqhat$ be the group of multiplicative characters on $\FF_q^\times$.  Extend any character $\chi$ on $\FF_q^\times$ to $\FF_q$ by defining $\chi(0) = 0$.  Let $\varepsilon$ denote the trivial character.  Following Greene \cite{greene}, for $x \in \FF_q$ and $\chi \in \FFqhat$, we define
\begin{equation}\label{delta-fcn}
\delta(\chi) := \begin{cases}
		1 & \text{ if $\chi = \varepsilon$,}\\
		0 & \text{ if $\chi \neq \varepsilon$,}
		\end{cases}
		\quad \text{ and } \quad
\delta(x) := \begin{cases}
		1 & \text{ if $x = 0$,}\\
		0 & \text{ if $x \neq 0$.}
		\end{cases}
\end{equation}

We will frequently have need of the following orthogonality relations for characters:
\begin{lemma}\label{orthogonality}
For all $A, B, \chi \in \FF_q^\times$ and all $x \in \FF_q$, we have
\begin{enumerate}
\item $\dsp \sum_{x \in \FF_q} A(x) B(x) = (q-1) \delta(AB).$
\medbreak

\item $\dsp \sum_{\chi \in \FF_q^\times} \chi(x) = (q-1) \delta(1-x).$
\end{enumerate}
\end{lemma}
%
%
For $A, B, \in \FFqhat$, define the {\em Jacobi sum}
\begin{equation}\label{jacobi-sum}
J(A,B) := \sum_{x \in \FF_q} A(x) B(1-x).
\end{equation}
Greene proved the following result, which is a finite field analogue of the binomial theorem:
\begin{theorem}[Theorem 2.3 of \cite{greene}]\label{FF-binom}
For any character $A \in \widehat{\FF_q^\times}$ and $x \in \FF_q$, we have
\begin{equation}\label{expand-char}
A(1-x) = \delta(x) + \frac{1}{q-1} \sum_{\chi \in \widehat{\FF_q^\times}} J(A, \overline{\chi}) \, \chi(x).
\end{equation}
\end{theorem}
%
%
This shows that the Jacobi sum may be viewed as an analogue of the binomial coefficient, so for $A, \, \chi \in \FFqhat$ we define
\begin{equation}\label{char-binomial}
\CC{A}{\chi} := -\chi(-1) \, J(A, \overline{\chi})
\end{equation}
This definition, given in \cite{dictionary}, differs from Greene's version by a factor of $-q$.  With this notation, the result of Theorem \ref{FF-binom} may be written as
\begin{equation}
A(1-x) = \delta(x) - \frac{1}{q-1} \sum_{\chi \in \widehat{\FF_q^\times}} \CC{A}{\chi} \, \chi(-x).
\end{equation}
%
%
Greene gives another version of Theorem \ref{FF-binom} which we require in proving certain tranformation formulas in Section \ref{section:trans}. 
\begin{theorem}[Equation (2.11) of \cite{greene}] For all $A, B, \chi \in \widehat{\FF_q}$ and all $x \in \FF_q$, we have 
$$\overline{B}(x) \overline{A}B(1-x) = \frac{q}{q-1} \sum_{\chi \in \widehat{\FF_q}} \ch{A\chi}{B\chi} \chi(x).$$
\end{theorem}
Converting this theorem to Jacobi sums, using our alternate definition of the binomial coefficients, we obtain the following property
\begin{equation}\label{alt-binom}
\overline{B}(x) \overline{A}B(1-x) = \frac{B(-1)}{q-1} \sum_{\chi \in \widehat{\FF_q}} J(A\chi, \overline{B\chi})  \, \chi(-x).
\end{equation}
%
%
Several identities involving Jacobi sums will also be of frequent use.  Note that the binomial coefficient versions of these identities hold for our version as well as Greene's, since the factors of $-q$ will simply cancel.
\begin{lemma}[See (2.6)--(2.8) in \cite{greene}] \label{jacobi-identities}  For any characters $A, B, C \in \widehat{\FF_q^\times}$,
\begin{enumerate}
\item $J(A, \overline{B}) = A(-1) \, J(A, B\overline{A})$, or $\dsp \CC{A}{B} = \CC{A}{A\overline{B}}$.
\medbreak

\item  $J(A, \overline{B}) = B(-1) \, J(B \overline{A},\overline{B})$, or $\dsp \CC{A}{B} = B(-1) \, \CC{B\overline{A}}{B}$.
\medbreak

\item $\dsp \CC{A}{B} = AB(-1) \CC{\overline{B}}{\overline{A}}$.
\medbreak

\item $J(A, \overline{B}) J(C, \overline{A}) = B(-1) J(C, \overline{B}) J(\overline{B}C, \overline{A}B) - \delta(A)(q-1) + \delta(B\overline{C})(q-1).$
\end{enumerate}
\end{lemma}


\section{Classical Appell-Lauricella functions}\label{section:classical}

As generalizations of the classical $\Ftwo$-hypergeometric series, Appell \cite{appell}, \cite{appell2}, \cite{appell-kampe} introduced four two-variable hypergeometric series, each with a different type of coupled Pochhammer symbol coefficients.  These four series were later generalized to several variables by Lauricella \cite{lauricella}.  We require the Appell functions of the first type, defined for $|x| < 1, |y| < 1$ by
%
%
\begin{equation}\label{Appell-def}
F_1(a;\, b_1,\, b_2;\, c \, |\, x,\,y) = \sum_{m,n = 0}^\infty \frac{\poch{a}{m+n} \, \poch{b_1}{m} \, \poch{b_2}{n}}{\poch{c}{m+n}\, m!\, n!} \; x^m y^n.
\end{equation}

Lauricella's series of type D give a natural generalization of $F_1$ to $n$ variables and are closely related to generalized Picard curves. Following the literature, we refer to these generalizations as Appell-Lauricella functions, as defined below.  For a comprehensive survey of Appell-Lauricella functions, we refer the reader to the article by Schlosser \cite{schlosser}, and to the monograph by Slater \cite{slater}. 

\begin{definition}\label{Lauricella-FD}
For $n \geq 2$, $a, c, b_1, b_2, \ldots b_n \in \C$, and $|x_1|<1, \ldots, |x_n| < 1$, the {\em Appell-Lauricella function} $\FDn$ is given by
\begin{equation*}\label{Lauricella-fcn}
F^{(n)}_D\left[\begin{matrix}a; & b_1 \,\,\,\,  \ldots \,\, \,\, b_n \smallskip \\  {} & c \end{matrix} \; ; \; x_1, \ldots, x_n \right]
:= \sum_{i_1, \ldots, i_n=0}^\infty \frac{\poch{a}{i_1 + \cdots i_n} \, \poch{b_1}{i_1} \cdots \poch{b_n}{i_n}}{\poch{c}{i_1 + \cdots i_n} \, i_1! \cdots i_n!} \; x_1^{i_1} \cdots x_n^{i_n}.
\end{equation*}
\end{definition}

We note that $F_1 = \FDtwo$.  For consistency we use the $\FDn$ notation throughout this paper.
%
%
Parallel to the classical one-variable setting, the $\FDn$-functions have the following one-variable integral representation, due to Picard when $n=2$, and Lauricella in the general case.

\begin{theorem}[Picard \cite{picard},  Lauricella \cite{lauricella}]\label{Appell-int}
For $\rm{Re}(c) > \rm{Re}(a) > 0$ and $|x_1|<1, |x_2|<1, \ldots, \\|x_n| < 1$, we have
\begin{align*}\label{Lauricella-int}
\Ln{n}{a}{b_1 & \ldots & b_n}{& c}{x_1, \ldots, x_n}
= \frac{\G(c)}{\G(a) \G(c-a)} \int_0^1 t^{a-1} (1-t)^{c-a-1} (1-x_1t)^{-b_1} \cdots (1-x_nt)^{-b_n} \, dt.\\
\end{align*}
\end{theorem}
%
%
From this we see that when $F_1 = F_D^{(n)}$ is evaluated at distinct parameters $\lam_1, \ldots, \lam_n \neq 0$ or $1$, it is naturally related to a period of the generalized Picard curve 
$$\ClamNijk:\,\, y^N = x^i (1-x)^j (1-\lam_1 x)^{k_1} \cdots (1-\lam_n x)^{k_n},$$ 
as defined in \eqref{C-lam}.
We define the following $n$-variable {\em period functions} $\PDn$, in order to demonstrate this relationship.
\begin{definition}\label{Appell-PD} For $n \geq 2$ and arbitrary constants $a, c, b_1, b_2, \ldots b_n \in \C$, the {\em period functions} $\PDn$corresponding to the Appell-Lauricella functions are given by
\begin{align*}
\LnP{n}{a}{b_1 & \ldots & b_n}{& c}{x_1, \ldots, x_n}
	&:= \int_0^1 t^{a-1} (1-t)^{c-a-1}
	    (1-x_1t)^{-b_1} \cdots (1-x_nt)^{-b_n} \, dt.
\end{align*}
\end{definition}

The period functions are a suitably-chosen normalization of the classical hypergeometric functions.  Specifically, Definition \ref{Appell-PD} immediately implies the following
\begin{align*}
\LnP{n}{a}{b_1 & \ldots & b_n}{& c}{x_1, \ldots, x_n}
=\frac{\Gamma(a) \Gamma(c-a)}{\Gamma(c)}
\Ln{n}{a}{b_1 & \ldots & b_n}{& c}{x_1, \ldots, x_n}.
\end{align*}
	
The Appell-Lauricella functions satisfy many transformation and reduction properties that are analogous to those satisfied by the classical hypergeometric functions.  For instance, when $n=2$, we have the following properties which we generalize to the finite-field setting in Section \ref{section:trans}.    
%
%
		
\begin{proposition}[Analogue of Pfaff-Kummer transformation,  (8.3.2) in \cite{slater}] \label{pfaff-analogue}
For $a, b_1, b_2, c \in \C$ and all $x, y$ for which the series are defined,
\begin{align*}
\Ln{2}{a}{b_1 & b_2}{c}{x, y}
= (1-x)^{-b_1} \, (1-y)^{-b_b} \,
\Ln{2}{c-a}{b_1 & b_2}{c}{\frac{x}{x-1}, \frac{y}{y-1}}.
\end{align*}
\end{proposition}
%
%
\begin{proposition}[Analogue of Euler's transformation,  (8.3.6) in \cite{slater})] \label{schlosser28}
\begin{align*}
\Ln{2}{a}{b_1 & b_2}{c}{x,\, y}
= (1-x)^{c-a-b_1} (1-y)^{-b_2}
\Ln{2}{c-a}{c-b_1-b_2 & b_2}{c}{x,\, \frac{x-y}{1-y}}.
\end{align*}
\end{proposition}
%
%
\begin{proposition}[Reduction Formulas,  (8.3.1.1) in \cite{slater}]\label{schlosser33ab}
\begin{align*}
\Ln{2}{a}{b_1 & b_2}{c}{x, \, x}
&= (1-x)^{c-a-b_1-b_2} \,
\pFq{2}{1}{c-a &  c-b_1-b_2}{& c}{x},\\
&= \pFq{2}{1}{a &  b_1+b_2}{& c}{x}.
\end{align*}
\end{proposition}	
Note that the second equality in Proposition \ref{schlosser33ab} follows immediately from the integral representation in Theorem \ref{Appell-int}, and Euler's integral transformation formula for ${}_2F_1$.	


\section{Finite-field Appell-Lauricella functions}\label{section:ff}	
	
In this section, we define finite-field analogues of the Appell-Lauricella period functions $\PDn$ and hypergeometric functions $\FDn$ given in Section \ref{section:classical}.  This parallels the one-variable construction in \cite{dictionary}.  There, the authors defined period functions $\PPn$, and these were suitably normalized to give hypergeometric functions $\FFn$.  
%
%
\begin{definition}\label{PPDn} 
For $n \geq 2$ and $A, C, B_1, \ldots B_n \in \FFqhat$, we define the {\em Appell-Lauricella period functions} $\PPDn$ over $\FF_q$ by
\begin{multline*}
\FAPn{n}{A}{B_1& \cdots & B_n}{& C}{\lam_1, \ldots, \lam_n}
	:= \sum_{y \in \FF_q} A(y) C\overline{A}(1-y) \,  \overline{B_1}(1-\lambda_1 y)\,  \overline{B_2}(1-\lambda_2 y) \cdots \overline{B_n}(1-\lambda_n y).
\end{multline*}
\end{definition}

Note that this definition is symmetric in the characters $B_1, \ldots, B_n$.  Also note that when $n=1$, this definition recovers the ${}_2\mathbb P_1$-period function in \cite{dictionary}, but with the first two parameters reversed.  That is,
$$
  \FAPn{1}{A}{B}{ C}{\lambda}=\pPPq21{B&A}{& C}{\lam}.
$$
(Although the ${}_2\mathbb P_1$-function is not symmetric in $A$ and $B$, the ${}_2\mathbb F_1$-function is symmetric in these parameters.)  Moreover, when $n\geq 2$, if $\lam_i = 0$ or $1$ for some $i$, the $\PPDn$-function reduces to a 
$\PP_D^{(n-1)}$-function in the remaining variables.  Without loss of generality (by symmetry in the $B_i$), 
if $\lambda_n = 0$ then $B_n(1-\lam_n y) = 1$, and we have
\begin{align*}
 \FAPn{n}A{B_1&\cdots &B_n}{&C}{\lambda_1, \ldots, \lambda_{n-1}, 0\, }= \FAPn{n-1}A{B_1&\cdots &B_{n-1}}{&C}{\lam_1, \ldots, \lam_{n-1}.}
\end{align*}
Thus when each $\lambda_i = 0$, the period function simply reduces to the Jacobi sum $J(A, C\overline{A})$.
 
Similarly, when $\lambda_n=1$, then $B_n(1-\lam_n y) = B_n(1-y)$ and
when $c=0$ or $1$, we therefore have
\begin{align*}
 \FAPn{n}A{B_1&\cdots &B_n}{&C}{\lambda_1, \ldots, \lambda_{n-1}, 1\, }= \FAPn{n-1}{A}{B_1&\cdots &B_{n-1}}{&CB_n}{\lam_1, \ldots, \lam_{n-1} }
\end{align*} 
Hereafter, we assume that $\lam_i\neq 0$ or $1$, whenever $n\geq 2$.

As is done in the case of the classical Appell-Lauricella hypergeometric functions, we normalize the period function $\PPDn$ by dividing out its value when all $\lam_i=0$, and we define the Appell-Lauricella hypergeometric functions $\FFDn$ as follows.
%
%
\begin{definition}\label{FFDn}
For $n \geq 2$ and $A, C, B_1, B_2, \ldots B_n \in \FFqhat$, we define the {\em Appell-Lauricella hypergeometric functions} over $\FF_q$ by
\begin{multline*}
\FAFn{n}{A}{B_1& \cdots & B_n}{& C}{\lam_1, \ldots, \lam_n}
:= \frac{1}{J(A, C\overline{A})} \cdot
\FAPn{n}{A}{B_1& \cdots & B_n}{& C}{\lam_1, \ldots, \lam_n}.
\end{multline*}
\end{definition}

We also give expressions for the $\PPDn$-period functions in terms of Jacobi sums.  
%
%
\begin{proposition}[Jacobi sum Representations for Period Functions]\label{PD-binom}
The following identities hold.
\begin{enumerate}
\item  (Equation 48 in \cite{dictionary}) For $A$, $B$, $C \in \FFqhat$, and $\l \in \FF_q$, 
%
%
\begin{align*}
  \FAPn{1}{A}{B}{C}{\lambda}=&\frac{A(-1)}{q-1}\sum_{\chi \in \FFqhat} J(B\chi, \overline{\chi}) \, J(A\chi, \overline{C\chi})\chi(\lambda)
    +\delta(\lam)J(A,C\ol A)\\
    =&\frac{AC(-1)}{q-1}\sum_{\chi \in \FFqhat} \CC{B\chi}{\chi}\CC{A\chi}{C\chi}\chi(\lambda)
    +\delta(\lam)J(A,C\ol A)
\end{align*}
%
%
\item
For $A, B_1, \cdots, B_n, C \in \FFqhat$ and $\lambda_1, \cdots, \lambda_n \in \FF_q^\times$,
\begin{align*}
  &\FAPn{n}{A}{B_1&\cdots &B_n}{& C}{\lambda_1, \ldots, \lambda_n }\\
    &=\frac{A(-1)}{(q-1)^n}\sum_{\chi_1,\ldots,\chi_n \in \FFqhat}J(A\chi_1\chi_2\cdots\chi_n, \overline{C\chi_1\chi_2\cdots\chi_n}) \(\prod_{i=1}^n J(B_i\chi_i, \overline{\chi_i})\chi_i(\lambda_i)\)  \\
    &=(-1)^{n+1}\frac{AC(-1)}{(q-1)^n} \sum_{\chi_1,\ldots,\chi_n \in \FFqhat}\CC{A\chi_1\chi_2\cdots\chi_n}{ C\chi_1\chi_2\cdots\chi_n}\(\prod_{i=1}^n \CC{B_i\chi_i}{\chi_i}\chi_i(\lambda_i)\)  \\
\end{align*}
\end{enumerate}
\end{proposition}
\begin{proof}

Applying \eqref{expand-char} to the character $\overline{B}$, we see that when $\lam_i y \neq 0$ we have
$$
  \overline{B_i}(1-\lambda_i y)=\frac{1}{q-1} \sum_{\chi_i \in \widehat{\FF_q^\times}} J(\ol B_i, \overline{\chi_i}) \, \chi_i(\lam_i y)
  =\frac{1}{q-1} \sum_{\chi_i \in \widehat{\FF_q^\times}} J( B_i\chi_i, \overline{\chi_i}) \, \chi_i(-\lam_i y).
$$

Using this fact along with Definition \ref{PPDn}, we have

\begin{align*}
  &\FAPn{n}{A}{B_1&\cdots &B_n}{& C}{\lambda_1, \ldots, \lambda_n }\\
    &=\frac{1}{(q-1)^n}\sum_{\chi_1,\ldots,\chi_n \in \FFqhat}J(A\chi_1\chi_2\cdots\chi_n, C\ol A) \(\prod_{i=1}^n J(B_i\chi_i, \overline{\chi_i})\chi_i(-\lambda_i)\)  \\
    &=\frac{A(-1)}{(q-1)^n}\sum_{\chi_1,\ldots,\chi_n \in \FFqhat}J(A\chi_1\chi_2\cdots\chi_n,\ol{C \chi_1\chi_2\cdots\chi_n}) \(\prod_{i=1}^n J(B_i\chi_i, \overline{\chi_i})\chi_i(\lambda_i)\)  
\end{align*}
The desired identity follows by applying \eqref{char-binomial}.
\end{proof}
%
%
The case $n=2$ is used often in Section \ref{section:trans}, in proving transformation properties for the $\PPDtwo$-functions.  We state it separately here for convenience.

\begin{corollary}\label{prop: PD2-binom} 
For $A, B_1, B_2$, $C \in \FFqhat$ and $\lambda_1$, $\lam_2 \in \FF_q$, we have
\begin{equation*}
\begin{split}
\FAPn{2}A{B_1&B_2}{C}{\lambda_1, \lam_2 }
&= \frac{-AC(-1)}{(q-1)^2} \sum_{\chi, \psi \in \FFqhat} \CC{B_1\chi}{\chi} \CC{B_2\psi}{\psi} \CC{A\chi \psi}{C\chi \psi} \chi(\lambda_1) \psi(\lam_2)\\
 &\qquad + {\delta(\lam_2)} \frac{AC(-1)}{q-1} \sum_{\chi \in \FFqhat} \CC{B_1\chi}{\chi} \CC{A\chi}{C\chi} \chi(\lambda_1)\\
 &\qquad + {\delta(\lambda_1)} \frac{AC(-1)}{q-1} \sum_{\chi \in \FFqhat} \CC{B_2\chi}{\chi} \CC{A\chi}{C\chi} \psi(\lam_2)\\
 &\qquad + \delta(\lambda_1) \delta(\lam_2) J(A, C\overline{A}).
\end{split}
\end{equation*}
\end{corollary}


\section{Geometric Interpretation}\label{section:geometric}

\subsection{Generalized Picard curves}
The integral representation of the period function $P_D^{(n)}$ can be naturally viewed in terms of a period of the smooth model of the generalized Picard curve given in \eqref{C-lam} and restated here for convenience
\begin{equation*}
\ClamNijk \; : \; y^N = x^i (1-x)^j (1-\lam_1 x)^{k_1} \cdots (1-\lam_n x)^{k_n}.
\end{equation*}
  The curve $\ClamNijk$ has singularities when $x = 0, 1, \infty,$ and $\frac{1}{\lambda_i}$, for $1 \leq j \leq n$.  Let 
$\XlamNijk$ denote its desingularization, or smooth model.    

%
%
%

To compute the genus of $\XlamNijk$, we utilize the following theorem of Archinard \cite{archinard}.

\begin{theorem}[[Theorem 4.1 of \cite{archinard}]\label{thm-archinard}
Let $X_N$ be the desingularization of the irreducible projective algebraic plane curve $C_N$ defined over $\C$ by the affine equation
\begin{equation}\label{CN-curve}
y^N = \prod_{i = 0}^r (x - \lambda_i)^{A_i},
\end{equation}
with $\lambda_0, \ldots, \lambda_r \in \C$ such that $\lambda_i \neq \lambda_j$ for all $i \neq j \in \{0, \ldots, r\}$.  Let $N, A_0, \ldots, A_r \in \N$ satisfy
$$N \neq \sum_{k=0}^r A_K \;\; \text{ and } \;\; \gcd(N, A_0, \ldots, A_r) = 1.$$
Then the Euler characteristic of $X_N(\C)$ is given by
$$\chi(X_N(\C)) = -rN + \gcd(N, N-\sum_{k=0}^r A_k) + \sum_{j=0}^r \gcd(N, A_j),$$
and the genus of $X_N$ by
\begin{equation}\label{XN-genus}
g[X_N] = (X_N(\C)) = 1 + \frac{1}{2}(rN - \gcd(N, N- \sum_{k=0}^rA_k) - \sum_{j=0}^r \gcd(N, A_j)).
\end{equation}
\end{theorem}

Our curves $\XlamNijk$ are isomorphic to the curves $X_N$ given in Theorem \ref{thm-archinard}, and so by applying this theorem, we find that the genus of $\XlamNijk$ is given by
%
%
\begin{equation}\label{genus-Xlam}
g(\XlamNijk) = 1 + \frac{1}{2} \bigg((n+1)N -\gcd(N, i+j+\sum_{j=1}^n k_j) - \gcd(N,i) - \gcd(N,j) - \sum_{j=1}^n\gcd(N,k_j) \bigg).
\end{equation}
In particular, the classical Picard curve, $C_{\lam_1, \lam_2}: \,\, y^3 = x^2(1-x)(1-\lam_1 x)(1-\lam_2 x)$, has genus 3.
\bigbreak

\subsection{Counting points on $\ClamNijk$ over finite fields}
Finite-field hypergeometric functions have been used by many authors in recent years to count points on affine hypersurfaces and algebraic varieties, with a number of applications. (See \cite{ahlono1}, \cite{ahlono2}, \cite{FOP}, \cite{dictionary}, \cite{mccarthy2}, \cite{vega}, \cite{WIN3b}, among others.) 
 
Here, we use a technique similar to the one used by Fuselier, et. al. in \cite{dictionary} to count points on generalized Legendre curves, which in turn is based on the point-counting method used by Vega in \cite{vega}. We first require the following well-known result on character sums.

%
%
\begin{lemma}[Proposition 8.1.5 of \cite{ireland-rosen}]\label{chisum}
Let $p$ be a prime and $a \in \FF_p \setminus \{0\}$.  If $n \mid (p-1)$ then
\begin{equation*}
\# \{x \in \FF_p \mid x^n = a\} = \sum_{\chi^n = \varepsilon} \, \chi(a),
\end{equation*}
where the sum runs over all characters $\chi \in \widehat{\FF_p^\times}$ of order dividing $n$.
\end{lemma}

%
%
Let $\#\ClamNijk(\FF_p)$ denote the number of $\FF_p$-points on the generalized Picard curve $\ClamNijk$.  We have the following formula in terms of Appell-Lauricella period functions.
\begin{theorem}\label{Picard-pts}
Let $p \equiv 1 \pmod{N}$ be a prime, and let $\eta_N \in \widehat{\FF_p^\times}$ be a primitive $N^{th}$-order character.  Then
\begin{align*}
\#\ClamNijk(\FF_p) = 1 + p \,+\,
\sum_{m=1}^{N-1}
\FAPn{n}{\eta_N^{mi}}{\eta_N^{-mk_i} & \cdots & \eta_N^{-mk_n}}{&  \eta_N^{mi+mj}}{\lambda_1, \ldots, \lambda_n }.
\end{align*}
\end{theorem}

\begin{proof}
Let $f(x,\mathbf{\lam}) = x^i(1-x)^j(1-\lambda_1 x)^{k_1}\cdots (1-\lambda_n x)^{k_n}$.   With $1$ representing the point at infinity, we have

\begin{equation*}
\begin{split}
\#\ClamNijk&(\FF_p) = 1 + \sum_{x \in \FF_p} \#\{ y \in \FF_p \mid y^N = f(x,\mathbf{\lam})\}, \\
&= 1 + \#\{ x \in \FF_p \mid f(x) = 0\} + \sum_{\substack{x \in \FF_p\\ f(x) \neq 0}}  \#\{ y \in \FF_p \mid y^N = f(x,\mathbf{\lam})\},\\
&= 1 + \#\{ x \in \FF_p \mid f(x) = 0\} + \sum_{x \in \FF_p}  \sum_{m=0}^{N-1} \, \eta_N^m(f(x,\mathbf{\lam})),\,\,\text{ by Lemma \ref{chisum}}, \\
&= 1 + \bigg[\#\{ x \in \FF_p \mid f(x) = 0\} + \sum_{x \in \FF_p}\, \varepsilon(f(x))\bigg] + \sum_{x \in \FF_p}  \sum_{m=1}^{N-1} \, \eta_N^m(f(x,\mathbf{\lam})).
\end{split}
\end{equation*}
The terms in the bracket combine to give $p$ as follows:   Since $\varepsilon(0)=0$, then for each $x \in \FF_p$, either $f(x)=0$ and this contributes $1$ to the count given by the first term in the bracket,  or $f(x) \neq 0$ and this contributes $1$ to the count given by the second term.  Therefore,
\begin{equation*}
\begin{split}
\#\ClamNijk&(\FF_p)
= 1 + p +   \sum_{m=1}^{N-1} \sum_{x \in \FF_p}  \, \eta_N^m(f(x,\mathbf{\lam})), \\
&= 1 + p +   \sum_{m=1}^{N-1} \sum_{x \in \FF_p}  \, \eta_N^m(x^i(1-x)^j(1-\lambda_1 x)^{k_1}\cdots (1-\lambda_n x)^{k_n}), \\
&= 1 + p +   \sum_{m=1}^{N-1} \sum_{x \in \FF_p}  \, \eta_N^{mi}(x)\eta_N^{mj}(1-x)\eta_N^{mk_1}(1-\lambda_1 x)\cdots \eta_N^{mk_n}(1-\lambda_n x),\\
&= 1 + p \,+\,
\sum_{m=1}^{N-1}
\FAPn{n}{\eta_N^{mi}}{\eta_N^{-mk_i} & \cdots & \eta_N^{-mk_n}}{&  \eta_N^{mi+mj}}{\lambda_1, \ldots, \lambda_n }\end{split}
\end{equation*}
\end{proof}

Counting $\FF_p$-points on the smooth model $\XlamNijk$ requires resolving the singularities at $0, 1, \infty,$ and $\frac{1}{\lambda_1}, \ldots, \frac{1}{\lambda_n}$, and determining the contribution arising from each one.  The contributions arising from these singularities affect the polynomial part $1+p$ in the count above, but more importantly, they do not change the hypergeometric functions that appear.  For brevity's sake, we omit the details here.  


\section{Some Degree $1$ Reduction and Transformation Formulas}\label{section:trans}

In this section, we consider the behavior of the $\PPDn$ and $\FFDn$ functions under transformations in the variables $\lambda_i$, as well as the simplifications that result when some of the characters are trivial or are repeated.  The following results are finite-field analogues of the classical results found in Slater \cite{slater}.  This development mirrors the translation of identities for classical hypergeometric functions to the finite field setting (see \cite{greene} and Section 8 of \cite{dictionary}).

\subsection{Transformation Formulas}
%
%
In order to establish these expressions, we first recall some necessary facts about the one variable period functions $\PPn$ defined in \cite{dictionary}.  The following proposition from \cite{dictionary} is due to Greene (Theorem 4.4 of \cite{greene}) using slightly different notation. 

\begin{proposition}\cite{dictionary}\label{P2-trans}
For any characters $A$, $B$, $C \in \widehat{\FF_q^\times}$, and $\lam \in \FF_q$, we have
\begin{eqnarray}
\pPPq21 {A&B}{&C}\lam
    &=&ABC(-1) \overline C(\lam)\, \pPPq21 {\, \overline CB&\overline CA}{&\overline C}\lam+\delta(\lam)J(B,C\overline B), \label{propeqn:first} \\
\pPPq21 {A&B}{&C}\lam
    &=&ABC(-1)\overline A(\lam )\, \pPPq21 {A&\overline CA}{&\overline BA}{\frac{1}{\lam}}+\delta(\lam)J(B,C\overline B), \label{propeqn:second} \\
\pPPq21 {A&B}{&C}\lam
   & =& B(-1)\, \pPPq21 {A&B}{&AB\overline C}{1-\lam}.  \label{propeqn:third}
\end{eqnarray}
\end{proposition}
%
%
The behavior of $\PPDn$ under the transformation $\dsp \lam_i \mapsto 1 - \lam_i$ for each $i$ is given in the following result, analogous to equation \eqref{propeqn:third} in Proposition \ref{P2-trans}.
	
\begin{proposition}\label{PDx->1-x} For $A, B_1, B_2, \cdots, B_n, C \in \widehat{\FF_q^\times}$ and $\lambda_1, \cdots, \lambda_n \in \FF_q$, we have
\begin{multline*}
\FAPn{n}A{B_1&\cdots &B_n}{&C}{\lambda_1, \ldots, \lambda_n }\\
= A(-1)\FAPn{n}A{B_1&\cdots&B_n}{& A\overline{C}B_1B_2 \cdots B_n}{1-\lambda_1, \ldots, 1-\lambda_n }.
\end{multline*}
\end{proposition}

\begin{proof}

In Definition \ref{PPDn}, make the change of variables $y \mapsto \frac{z}{z-1}$.  It follows that $1-y \mapsto \frac{1}{1-z}$ and
$1- \lambda_i y \mapsto \frac{1 - (1-\lambda_i) z}{1-z}$, and carrying out these substitutions gives
\begin{align*}
&\FAPn{n}A{B_1&\cdots&B_n}{&C}{\lambda_1, \ldots, \lambda_n }\\
&\qquad = \sum_{y \in \FF_q} A\left(\frac{z}{z-1}\right) C\overline{A}\left(\frac{1}{1-z}\right) \,  \overline{B_1}\left( \frac{1-(1-\lambda_1)z}{1-z}\right)\,  \cdots \overline{B_n}\left( \frac{1-(1-\lambda_n)z}{1-z}\right).\\
&\qquad = A(-1) \sum_{y \in \FF_q}\,A(z) \,\overline{A}B_1\cdots B_n \overline{C}A(1-z)\, \overline{B_1}(1-(1-\lambda_1)z) \cdots  \overline{B_n}(1-(1-\lambda_n)z)\\
&\qquad = A(-1)\FAPn{n}A{B_1&\cdots&B_n}{& A\overline{C}B_1B_2 \cdots B_n}{1-\lambda_1, \ldots, 1-\lambda_n }
\end{align*}
\end{proof}

The corresponding result for $\mathbb F_D^{(n)}$ follows as an immediate consequence.

\begin{corollary}\label{FDx->1-x} For $A, B_1, B_2, \cdots, B_n, C \in \widehat{\FF_q^\times}$ and $\lambda_1, \cdots, \lambda_n \in \FF_q$, we have
\begin{multline*}
\FAFn{n}A{B_1&\cdots &B_n}{&C}{\lambda_1, \ldots, \lambda_n }\\
= \frac{J(A,\ol C B_1B_2\ldots B_n)}{A(-1)J(A,C\ol A)}\FAFn{n}A{B_1&\cdots&B_n}{A\overline{C}B_1B_2 \cdots B_n}{1-\lambda_1, \ldots, 1-\lambda_n }
\end{multline*}
\end{corollary}
%
%
We next consider the behavior of $\PPDn$ under the transformation $ \lambda_i \mapsto \frac{1}{\lambda_i}$ for each $i$.  The following result is analogous to   equation \eqref{propeqn:second} \rmhs in Proposition \ref{P2-trans}.

\begin{proposition}\label{PDx->1/x}
The following identities hold,
\begin{enumerate}
\item 
For $A, B, C \in \widehat{\FF_q^\times}$ and $\lambda\in \FF_q$,
$$
\FAPn{1}AB{C}{\lambda}
= AC(-1)\ol{B}(-\lam)\FAPn{1}{\, \ol CB}{B}{\ol AB}{\frac 1{\lambda}}+\delta(\lam)J(A,C\ol A).
$$
\item
If $\lambda_i\in \FF_q^\times$, and $A$, $C$, $B_i \in \widehat{\FF_q^\times}$, then
  \begin{multline*}
\FAPn{n}A{B_1&\cdots&B_n}{&C}{\lambda_1, \ldots, \lambda_n }\\
= AC(-1)\(\prod_{i=1}^n\ol{B_i}(-\lam_i)\)\FAPn{n}{\, \ol C \, \prod_{i=1}^nB_i}{B_1&\cdots&B_n}{& \ol A \,  \prod_{i=1}^nB_i}{\frac 1{\lambda_1}, \ldots, \frac1{\lambda_n} }.
\end{multline*}
\end{enumerate}
\end{proposition}

\begin{proof}
  This follows directly from Proposition \ref{PD-binom}. 
\end{proof}
%
%
Recall the two-variable analogue of the classical Pfaff-Kummer transformation, which is given in Proposition \ref{pfaff-analogue}.
The following transformation for $\PPDn$, when $ \lam_i \mapsto \frac{\lam_i}{\lam_i-1}$ for each $i$, gives a finite-field analogue of this result. 
	
\begin{proposition}\label{PDx->x/(x-1)}
The following identities hold.

\begin{enumerate}
\item
For $A, B, C \in \FFqhat$ and $\lambda\in \FF_q$,
\begin{align*}
\FAPn{1}AB{C}{\lambda}
=& \ol{B}(1-\lam)\FAPn{1}{\, \ol AC}{B}{C}{\frac {\lambda}{\lam-1}}+\delta(1-\lam)J(A,C\ol {AB})\\
=&\ol{A}(1-\lam)\FAPn{1}{A}{C\ol B}{C}{\frac {\lambda}{\lam-1}}+\delta(1-\lam)J(A,C\ol {AB}).
\end{align*}
\item
If $\lambda_i\in \FF_q$ with $\lam_i\neq 1$, and $A$, $C$, $B_i \in \FFqhat$, then
  \begin{multline*}
\FAPn{n}A{B_1&\cdots&B_n}{&C}{\lambda_1, \ldots, \lambda_n }\\
= \(\prod_{i=1}^n\ol{B_i}(1-\lam_i)\)\FAPn{n}{\, \ol AC }{B_1&\cdots&B_n}{&C}{\frac {\lam_1}{\lambda_1-1}, \ldots, \frac{\lam_n}{\lambda_n-1} }.
\end{multline*}
\end{enumerate}
\end{proposition}

\begin{proof}
  
When $\lambda_i \neq 1$, the right-hand side is simply
\begin{align*}
\(\prod_{i=1}^n\ol{B_i}(1-\lam_i)\)&\FAPn{n}{\, \ol AC }{B_1&\cdots&B_n}{&C}{\frac {\lam_1}{\lambda_1-1}, \ldots, \frac{\lam_n}{\lambda_n-1} }\\
&=\(\prod_{i=1}^n\ol{B_i}(1-\lam_i)\)\sum_{y \in \FF_q} C\ol A(y) {A}(1-y) \, \prod_{i=1}^n \overline{B_i}\(1-\frac{\lam_i}{\lam_i-1} y\)\\
&= \sum_{y \in \FF_q} C\ol A(y) {A}(1-y) \, \prod_{i=1}^n \overline{B_i}\({1-\lam_i(1-y)}\)\\
&= \sum_{z \in \FF_q} A(z) C\overline{A}(1-z)\prod_{i=1}^n \overline{B_i}\({1-\lam_i z}\), \text{ after setting $y = 1-z$,}\\
&=\FAPn{n}A{B_1&\cdots&B_n}{&C}{\lambda_1, \ldots, \lambda_n }. 
\end{align*}
\end{proof}

\subsection{Reduction Formulas}
 
For the remainder of this section, we focus on translating several reduction formulas for the classical Appell functions into the finite field setting.  For simplicity, we restrict to the case $n=2$. First, we recall some useful additional properties of the $\PPtwo$ functions.

%
%
The first property is an analogue to Euler's formula for the classical $_2F_1$-hypergeometric functions, which can be obtained from the classical Pfaff transformation.

\begin{proposition}\cite{dictionary}\label{prop: double Pfaff}
For $A, B, C \in \widehat{\FF_q^\times}$,    we have
\begin{align*}
 \pPPq21 {A&B}{&C}\lambda &= C\ol {AB}(1-\lam)\pPPq21 {C\ol A&C\ol B}{&C}{\lam}+\delta(1-\lam)J(B,C\ol{AB}).
\end{align*}
\end{proposition}

The next proposition highlights the symmetries that appear when $A,B\neq \eps$ and $A,B\neq C$.

\begin{proposition}\cite{dictionary}\label{prop: commute and conjugation}
If $A, B, C \in \widehat{\FF_q^\times}$,  $A,B\neq \eps$ and $A,B\neq C$,  then  we have the following.
\begin{enumerate}
\item[(1)]  In general, we have
\begin{align*}
J(A,\overline AC)\cdot \pPPq21 {A&B}{&C}\lambda &= J(B,\overline B C)\cdot \pPPq21 {B&A}{&C}\lambda, \\
\pFFq{2}{1}{A&B}{&C}\lambda &=
\pFFq 21{B&A}{&C}\lambda;
\end{align*}
\item[(2)] For $\l\neq 0$, $1$, we have
\begin{align*}
\pPPq21 {A&B}{&C}\lambda &=
\overline C(\lambda)C\overline{AB}(\lambda-1)\frac{J(B,C\ol B )}{J(A, C\ol A)}\pPPq21 {\ol A&\ol B}{&\ol C}\lambda, \\
 \pFFq{2}{1}{A&B}{&C}\lambda &=
\overline C(\lambda)C\overline{AB}(\lambda-1)\frac{J(\ol B,\ol C B)}{J( A, C\ol A)}\pFFq{2}{1} {\overline A &\overline B}{&\overline C}\lambda.
\end{align*}
\end{enumerate}
\end{proposition}
%
%
 
Our first reduction formula follows from the previous propositions by considering the case when $\lam_1=\lam_2$.  We obtain a finite field analogue of the classical reduction formula from the first equality in Proposition \ref{schlosser33ab}.

\begin{proposition}\label{FF-schlosser33a} For $A, B_1, B_2, C \in \widehat{\FF_q^\times}$ and $\lambda \in \FF_q$, if $A$, $B_1B_2\neq \eps$, then
\begin{align*}
\FAPn{2}A{B_1&B_2}{C}{\lambda, \lam }
=& C\overline{AB_1B_2}(1-\lambda)  \,
\frac{J(A,C\ol A)}{J(C\ol{B_1B_2}, B_1B_2)}\pPPq21{C\ol A& C\ol{B_1B_2}}{&C}\lam\\
&+\delta(1-\lambda) J(A, \overline{AB_1B_2}C),
\end{align*}
\begin{align*}
\FAFn{2}A{B_1&B_2}{C}{\lambda, \lam }
= C\overline{AB_1B_2}(1-\lambda)  \, \pFFq21{C\ol A& C\ol{B_1B_2}}{&C}\lam
+\delta(1-\lambda) \frac{J(A, \overline{AB_1B_2}C)}{J(A,C\ol A)}.
\end{align*}
\end{proposition}

\begin{proof}
By the definition, Proposition \ref{prop: double Pfaff} and Proposition \ref{prop: commute and conjugation}, we observe that
\begin{align*}
  \FAPn{2}A{B_1&B_2}{C}{\lambda, \lam }=&\pPPq21{ {B_1B_2}&A}{&C}\lam\\
  =&C\overline{AB_1B_2}(1-\lambda)\pPPq21{ C\ol{B_1B_2}&C\ol A}{&C}\lam+\delta(1-\lambda) J(A, \overline{AB_1B_2}C)\\
          =& C\overline{AB_1B_2}(1-\lambda)  \,
\frac{J(A,C\ol A)}{J(C\ol{B_1B_2}, B_1B_2)}\pPPq21{C\ol A& C\ol{B_1B_2}}{&C}\l\\
&+\delta(1-\lambda) J(A, \overline{AB_1B_2}C).
\end{align*}
The identity for $\FFDn$ then follows from the definition.
\end{proof}
%
%

In \cite{dictionary}, the  second and third authors, et. al., characterize the period functions $\PPn$ and their corresponding $\FFn$ hypergeometric functions  as {\em primitive} if $A_i \neq \varepsilon$ and $A_i \neq B_j$ for all $i,j$.  Similarly, we make the following definition.

\begin{definition}\label{primitive}
A period function 
${\mathbb P}^{(n)}_D\left[\begin{matrix}A;&B_1 \,\,\,\, \ldots \,\,\,\, B_n  \smallskip \\  {} & C \end{matrix} \; ; \; \lam_1, \ldots, \lam_n\right]$ 
or hypergeometric function
${\mathbb F}^{(n)}_D\left[\begin{matrix}A;&B_1 \,\,\,\, \ldots \,\,\,\, B_n  \smallskip \\  {} & C \end{matrix} \; ; \; \lam_1, \ldots, \lam_n\right]$ 
is said to be {\em primitive} if $A, B_i \neq \varepsilon$ and $A, B_i \neq C$ for all $i$.
\end{definition}
%
%

Following the strategies in the proof of Proposition 4 from \cite{dictionary}, we obtain the following reduction formulas for the $\PPDtwo$-period functions.
\begin{proposition}\label{PD-reductions}
For $A, B_1, B_2, C \in \widehat{\FF_q^\times}$, and $\lambda_1, \lambda_2 \in \FF_q$,
\begin{enumerate}\label{B1-trivial}
\item If $\lambda_1 \neq 0$, then
\begin{align*}
    \FAPn{2}A{\eps &B_2}{C}{\lambda_1, \lam_2 }=\pPPq21{ {B_2}&A}{&C}{\lam_2} - B_2\overline{C}(\lambda_1) \overline{A}C(\lambda_1-1) \overline{B_2}(\lambda_1 - \lambda_2),
\end{align*}
\bigbreak

\item If $\lambda_1 \neq 0$, then
\begin{align*}
    \FAPn{2}A{B_1 &B_2}{A}{\lambda_1, \lam_2 }=\ol A(\lam_1)\pPPq21{ {B_2}&A}{&AB_1}{\frac{\lam_2}{\lam_1}} - \overline{B_1}(1- \lambda_1) \overline{B_2}(1-\lambda_2) \overline{B_2}(\lambda_1 - \lambda_2),
\end{align*}
\bigbreak

\item If $\lam_2\neq 1$ then
\begin{align*}
    \FAPn{2}A{B &C\ol B}{C}{\lambda_1, \lam_2 } =& B\ol A(\lam_2-1)\ol B(\lam_2-\lam_1)\pPPq 21{B& B\ol C}{& AB\ol C}{\frac{1-\lam_1}{\lam_2-\lam_1}}\\
    &\qquad \qquad -A(-1)B\ol C(\lam_2)\ol B(\lam_1) +\delta(\lam_2-\lam_1)\ol A(\lam_1-1)J(A,\ol C),
\end{align*}
\bigbreak

\item If { $\lam_1\neq 0$, $1$, and $\lam_2\neq 1$,} then
\begin{align*}
    \FAPn{2}\eps{B_1 & B_2}{C}{\lambda_1, \lam_2 } = \,\,
{
C\ol B_1(1-\lam_1)\ol C(-\lam_1)\ol B_2(1-\lam_2)\pPPq 21{B_2& C}{& C\ol B_1}{\frac{\lam_2(1-\lam_1)}{\lam_1(1-\lam_2)}}-1.
}
\end{align*}
\end{enumerate}
\end{proposition}

\begin{proof}
Certain parts of the proof involve a straightforward use of character sums, while others rely on the Jacobi sum interpretation.
\begin{enumerate}
%
%
\item When $\lambda_1 \neq 0$, we have
\begin{align*}
\begin{split}
 \FAPn{2}A{\eps &B_2}{C}{\lambda_1, \lam_2 }
&= \sum_{y \in \FF_q} A(y) \overline{A}C(1-y) \varepsilon (1-\lambda_1 y) \overline{B_2}(1-\lambda_2 y)\\
&= \sum_{\substack{y \in \FF_q\\ y \neq 1/\lambda_1}} A(y) \overline{A}C(1-y) \overline{B_2} (1 - \lambda_2 y), \text{ by definition of $\varepsilon$}.
\end{split}
\end{align*}
Simplifying using properties of characters and the definition of the period function gives
\begin{align*}
\begin{split}
& \FAPn{2}A{\eps &B_2}{C}{\lambda_1, \lam_2 }\\
& \qquad \qquad = \sum_{y \in \FF_q} B(y) \overline{A}C(1-y) \overline{B_2}(1-\lambda_2 y) - A\left(\frac{1}{\lambda_1}\right) \overline{A}C\left(1 - \frac{1}{\lambda_1}\right)\overline{B_2}\left(1 - \frac{\lambda_2}{\lambda_1}\right)\\
 & \qquad \qquad =\pPPq21{ {B_2}&A}{&C}{\lam_2}
- \overline{A}(\lambda_1)\overline{A}C(\lambda_1-1) A\overline{C}(\lambda_1)\overline{B_2}(\lambda_1-\lambda_2) B_2(\lambda_1)\\
 & \qquad \qquad = \pPPq21{ {B_2}&A}{&C}{\lam_2}
- B_2\overline{C}(\lambda_1) \overline{A}C(\lambda_1-1)\overline{B_2}(\lambda_1-\lambda_2).
\end{split}
\end{align*}
\bigbreak
%
%
\item When $\lambda_1 \neq 0$, we have
\begin{align*}
\begin{split}
 \FAPn{2}A{B_1 &B_2}{A}{\lambda_1, \lam_2 }
&= \sum_{y \in \FF_q} A(y) \varepsilon(1-y) \overline{B_1} (1-\lambda_1 y) \overline{B_2}(1-\lambda_2 y)\\
&= \sum_{\substack{y \in \FF_q\\ y \neq 1}} A(y) \overline{B_1} (1-\lambda_1 y) \overline{B_2} (1 - \lambda_2 y)\\
&= \sum_{y \in \FF_q} A(y) \overline{B_1} (1-\lambda_1 y) \overline{B_2} (1 - \lambda_2 y) - \overline{B_1}(1-\lambda_1) \overline{B_2}(1-\lambda_2).  \\
\end{split}
\end{align*}
Now make the substitution $\dsp y = \frac{z}{\lambda_1}$, so that $\lambda_1 y = z$.  Thus
\begin{align*}
\begin{split}
 \FAPn{2}A{B_1 &B_2}{A}{\lambda_1, \lam_2 }
&= \sum_{z \in \FF_q} A \left(\frac{z}{\lambda_1}\right) \overline{B_1}(1-z) \overline{B_2} \left( 1 - \frac{\lambda_2 z}{\lambda_1}\right) - \overline{B_1}(1-\lambda_1) \overline{B_2}(1-\lambda_2)\\
&= \overline{A}(\lambda_1)  \sum_{z \in \FF_q} A(z) \overline{B_1}(1-z) \overline{B_2} \left( 1 - \frac{\lambda_2}{\lambda_1}z \right) - \overline{B_1}(1-\lambda_1) \overline{B_2}(1-\lambda_2)\\
&=\ol A(\lam_1)\pPPq21{ {B_2}&A}{&AB_1}{\frac{\lam_2}{\lam_1}}  - \overline{B_1}(1-\lambda_1) \overline{B_2}(1-\lambda_2).
\end{split}
\end{align*}
\bigbreak
%
%
\item {
From the definition, we know
\begin{align*}
  \FAPn{2}A{B &C\ol B}{C}{\lambda_1, \lam_2 }
= \sum_{y \in \FF_q} A\(\frac y{1-y}\)\ol C\(\frac{1-\lam_2 y}{1-y}\)\ol B\(\frac{1-\lam_1 y}{1-\lam_2 y}\).
\end{align*}
When $\lam_2\neq 1$, replace $y$ by $(1-y)/(\lam_2-y)$, so $(1-\lam_2 y)/(1-y)$ becomes $y$ and then
\begin{align*}
\begin{split}
 \FAPn{2}A{B &C\ol B}{C}{\lambda_1, \lam_2 }
&= \sum_{\substack{y \in \FF_q\\ y \neq \lam_2}} A\(\frac {1-y}{\lam_2-1}\)\ol C\(y\)\ol B\(\frac{\lam_2-\lam_1+(\lam_1-1) y}{(\lam_2-1) y}\),\\
&= \ol A B(\lam_2-1)\sum_{{y \in \FF_q}} B\ol C(y) A (1- y) \overline{B}\((\lam_2-\lam_1)-(1-\lam_1)y\)\\
&\qquad \qquad \qquad \qquad -A(-1)\ol B(\lam_1)B\ol C(\lam_2).
\end{split}
\end{align*}
Therefore, if $\lam_1\neq \lam_2$, we have
\begin{multline*}
  \FAPn{2}A{B C & \ol B}{C}{\lambda_1, \lam_2 }\\
= B\ol A(\lam_2-1)\ol B(\lam_2-\lam_1)\pPPq 21{B& B\ol C}{& AB\ol C}{\frac{1-\lam_1}{\lam_2-\lam_1}}
-A(-1)B\ol C(\lam_2)\ol B(\lam_1),
\end{multline*}
and if $\lam_1= \lam_2$, the reduced formula becomes
$$
  \pPPq 21{C& A}{& C}{\lam_2}=\ol A(\lam_2-1)J(A,\ol C)-A(-1)\ol C(\lam_2).
$$
}
\bigbreak
%
%
\item  {
Under the assumptions $\lam_1\neq 0$, $1$, and $\lam_2\neq 1$, we have
\begin{align*}
\begin{split}
  \FAPn{2}\eps{B_1 & B_2}{C}{\lambda_1, \lam_2 }
&= \sum_{y \in \FF_q} \eps(y) C(1-y) \overline{B_1} (1-\lambda_1 y) \overline{B_2}(1-\lambda_2 y)\\
&= \sum_{\substack{y \in \FF_q}} C(1-y) \overline{B_1} (1-\lambda_1 y) \overline{B_2} (1 - \lambda_2 y)-1 \, (\mbox{from } y=0) \\
&= \sum_{y \in \FF_q} C(y) \overline{B_1} (1-\lambda_1(1-y)) \overline{B_2} (1 - \lambda_2 (1-y)) - 1\\
&=\ol B_1(1-\lam_1)\ol B_2(1-\lam_2)\sum_{y \in \FF_q} C(y) \overline{B_1} \(1-\frac{\lam_1}{\lam_1-1}y\) \overline{B_2} \(1-\frac{\lam_2}{\lam_2-1}y\) - 1\\
\end{split}
\end{align*}
Now after making the substitution $\dsp y = \frac{\lam_1-1}{\lambda_1}z$, 
\begin{align*}
\begin{split}
&  \FAPn{2}\eps{B_1 & B_2}{C}{\lambda_1, \lam_2 }\\
&\qquad \qquad = C \ol B_1(1-\lam_1)\ol B_2(1-\lam_2)\ol C(-\lam_1)\sum_{z \in \FF_q} C(z) \overline{B_1}(1-z) \overline{B_2} \left( 1 - \frac{\lambda_2 }{1-\lambda_1}\frac{1-\lam_1}{\lam_1}z\right) - 1\\
&\qquad \qquad =  C \ol B_1(1-\lam_1)\ol B_2(1-\lam_2)\ol C(-\lam_1)\pPPq 21{B_2& C}{& C\ol B_1} {\frac{\lam_2(1-\lam_1)}{\lam_1(1-\lam_2)}} - 1.
\end{split}
\end{align*}
\medbreak

}
\end{enumerate}
\end{proof}


\section{Cubic Transformation Formulas}\label{section:cubic}
In this section we prove our main result, which is a cubic transformation for the two-variable finite-field Appell-Lauricella functions.  
We first prove Theorem \ref{finiteF1-cubic}, our finite-field analogue of Koike and Shiga's cubic transformation, given in Section \ref{section:intro} and restated here for convenience:

\newtheorem*{finiteF1-cubic}{Theorem \ref{finiteF1-cubic}}	
\begin{finiteF1-cubic}
Let $p\equiv 1 \pmod{3}$ be prime, let $\omega$ be a primitive cubic root of unity, and let $\eta_3$ be a primitive cubic character in $\widehat{\mathbb{F}_p^\times}$.  If $\lambda, \mu \in \FF_p$ satisfy $1 + \lambda + \mu \neq 0$,  then
\begin{align*}
\FAFn{2}{\eta_3}{\eta_3 &  \eta_3}{\eps}{1-\lambda^3, 1-\mu^3} = \FAFn{2}{\eta_3}{\eta_3 &  \eta_3}{\eps}{\left(\frac{1+\omega \lambda + \omega^2 \mu}{1 + \lambda + \mu}\right)^3, \left(\frac{1+\omega^2 \lambda + \omega \mu}{1 + \lambda + \mu}\right)^3}. 
\end{align*}
\end{finiteF1-cubic}

\begin{proof} For ease of notation, set
\begin{equation}\label{eqn:zetadef}
\zeta_1 = \zeta_1(\lambda, \mu) := \frac{1+\omega\lambda + \omega^2\mu}{1+\lambda+\mu}, \quad 
\zeta_2 = \zeta_2(\lambda, \mu) :=  \frac{1+\omega^2\lambda + \omega\mu}{1+\lambda+\mu}.
\end{equation}
  Observe that by the definition of $\mathbb{F}_D^{(2)}$, the statement of Theorem \ref{finiteF1-cubic} is equivalent to the claim that 
\begin{equation}\label{cubic-equivalent}
\sum_{x\in\mathbb{F}_p} \eta_3^2\left( x^2(1-x) \left(1-(1-\lambda^3) x \right) \left(1-(1-\mu^3) x \right) \right) \\
= \sum_{x\in\mathbb{F}_p} \eta_3^2\left( x^2(1-x) (1-\zeta_1 ^3 x) (1-\zeta_2 ^3 x)  \right).
\end{equation}
Therefore, to prove Theorem \ref{finiteF1-cubic}, it suffices to prove equation \eqref{cubic-equivalent}.  To this end, 
let $\zeta_1, \zeta_2$ be as in \eqref{eqn:zetadef}, and define the polynomials 
\begin{align*}
f_{\lambda, \mu}(x) &:= x^2(1-x) (1-(1-\lambda^3) x)(1-(1-\mu^3) x),\\
g_{\lambda,\mu}(x) &:= x^2(1-x) (1-\zeta_1^3 x)(1-\zeta_2^3 x).
\end{align*}
   Let $c\in \FF_p^\times$, and define curves $C_1:y^3=cf_{\lambda, \mu}(x)$ and $C_2:y^3=cg_{\lambda, \mu}(x)$.  If we show that $C_1$ and $C_2$ have the same trace of Frobenius, then by Theorem \ref{Picard-pts}, we may conclude that
\begin{equation}\label{eqn:sametrace}
\sum_{i=1}^2 \sum_{x\in\mathbb{F}_p} \eta_3^i(cf_{\lambda,\mu}(x)) = \sum_{i=1}^2 \sum_{x\in\mathbb{F}_p} \eta_3^i(c
g_{\lambda,\mu}(x)). 
\end{equation}
Putting $c=1$ in \eqref{eqn:sametrace} and multiplying through by $\omega$ will give 

\begin{equation}\label{eq:trace1}
\omega \sum_{x\in\mathbb{F}_p} \eta_3(f_{\lambda,\mu}(x)) + \omega \sum_{x\in\mathbb{F}_p} \eta_3^2(f_{\lambda,\mu}(x)) = \omega \sum_{x\in\mathbb{F}_p} \eta_3(g_{\lambda,\mu}(x)) + \omega \sum_{x\in\mathbb{F}_p} \eta_3^2(g_{\lambda,\mu}(x)).
\end{equation}
Since $\eta_3$ is a primitive cubic character, there exists $c\in\mathbb{F}_p^\times$ such that $\eta_3(c)=\omega$.    Using this choice of $c$ in \eqref{eqn:sametrace}, we see that
\begin{equation}\label{eq:trace2}
\omega \sum_{x\in\mathbb{F}_p} \eta_3(f_{\lambda,\mu}(x)) + \omega^2 \sum_{x\in\mathbb{F}_p} \eta_3^2(f_{\lambda,\mu}(x)) = \omega \sum_{x\in\mathbb{F}_p} \eta_3(g_{\lambda,\mu}(x)) + \omega^2 \sum_{x\in\mathbb{F}_p} \eta_3^2(g_{\lambda,\mu}(x)).
\end{equation}
 
Subtracting equation \eqref{eq:trace1} from equation \eqref{eq:trace2}, and then dividing by $\omega^2 - \omega$, we have
\[
\sum_{x\in\mathbb{F}_p} \eta_3^2(f_{\lambda,\mu}(x)) = \sum_{x\in\mathbb{F}_p} \eta_3^2(g_{\lambda,\mu}(x)). 
\]
(Note that an analous argument shows that $\sum_{x\in\mathbb{F}_p} \eta_3(f_{\lambda,\mu}(x)) = \sum_{x\in\mathbb{F}_p} \eta_3(g_{\lambda,\mu}(x).$)
Therefore, proving Theorem \ref{finiteF1-cubic} has been reduced to showing that the curves $C_1$ and $C_2$ have the same trace of Frobenius.  We establish this result below. 

For $i=1, 2$, let $a_p(C_i)$ denote the trace of Frobenius on $C_i$, hence $a_p(C_i) = 1+p - \#C_i(\FF_p)$. 
It suffices to show that the curves have the same trace of Frobenius for almost all primes.  Since $C_1$ and $C_2$ have genus $3$ by Theorem \ref{thm-archinard}, we see that for each curve, the trace of Frobenius is bounded by $6\sqrt{p}$.   Thus if we show that the traces of Frobenius are congruent modulo $p$ for all $p>36$, then they are in fact equal.  When $p\equiv 2\pmod{3}$, there are no primitive cubic characters in $\widehat{\mathbb{F}_p^\times}$, and so $a_p(C_1) = 0 = a_p(C_2)$. In light of Theorem \ref{Picard-pts}, it therefore suffices to show that for $p\equiv 1 \pmod{3}$, with $1+\lambda+\mu\neq 0$, we have
\[
\sum_{x\in\mathbb{F}_p} \eta_3(f_{\lambda,\mu}(x)) \equiv \sum_{x\in\mathbb{F}_p} \eta_3(g_{\lambda,\mu}(x)) \pmod{p}. 
\]
We prove this by showing that the Hasse invariants of $C_1$ and $C_2$ agree modulo $p$, for $p \equiv 1 \pmod{3}$.  This requires the following result of Matsumoto and Ohara \cite[Thm. 1]{matsumoto-ohara}.

\begin{theorem}[Matsumoto, Ohara \cite{matsumoto-ohara}]\label{thm:MO}
Let $\omega=e^{\frac{2\pi i}{3}}$, $c\in \Z$, and $(x,y)\in \C^2$ be in an appropriately small neighborhood of $(1,1)$. Then the Lauricella function $F_{D}^{(2)}$ satisfies the following transformation formula
\begin{multline*}
\left( \frac{1+x+y}{3} \right)^c 
\Ln{2}{\frac{c}{3}}{\frac{c+1}{6}& \frac{c+1}{6}}{\frac{c+1}{2}}{1-x^3, 1-y^3}\\
=\Ln{2}{\frac{c}{3}}{\frac{c+1}{6}& \frac{c+1}{6}}{\frac{c+5}{6}}{\left(\frac{1+\omega x + \omega^2 y}{1+x+y} \right)^3, \left(\frac{1+\omega^2 x + \omega y}{1+x+y} \right)^3}.
\end{multline*}
\end{theorem}

The Hasse invariant $H(C)$ of a curve $C: y^3=x^2(1-x)(1-sx)(1-tx)$ in parameters $s,t$ is the $(p-1)^{st}$ coefficient of the polynomial $(x^2(1-x)(1-sx)(1-tx))^\frac{p-1}{3}$.  Let $m=\frac{p-1}{3}\in\N$.  By the binomial theorem, we have that 
\[
(x^2(1-x)(1-sx)(1-tx))^m = \sum_{i,j,k=0}^m\binom{m}{i}\binom{m}{j}\binom{m}{k}(-1)^{i+j+k}s^jt^kx^{2m+(i+j+k)}.
\]
Thus the $(p-1)^{st}$ coefficient will be obtained when $j+k=m-i$.  That is, 
\begin{align*}
H(C) = (-1)^m \sum_{j=0}^m\sum_{k=0}^{m-j}\binom{m}{j+k}\binom{m}{j}\binom{m}{k}s^jt^k = (-1)^\frac{p-1}{3}
 \Ln{2}{\frac{1-p}{3}}{\frac{1-p}{3}& \frac{1-p}{3}}{1}{s, t}_{\frac{p-1}{3}}, 
\end{align*}

where the subscript $\frac{p-1}{3}$ denotes the truncation of the series at the degree $\frac{p-1}{3}$ term.

This implies that 
\begin{align*}
H(C) \equiv (-1)^\frac{p-1}{3}
 \Ln{2}{\frac{1}{3}}{\frac{1}{3}& \frac{1}{3}}{1}{s, t}_{\frac{p-1}{3}} \pmod{p}.
 \end{align*}
Thus to complete the proof, it remains to show that the $F_{D}^{(2)}$ functions corresponding to $C_1$ and $C_2$ agree modulo $p$, which we handle in the following lemma.
\end{proof}

\begin{lemma}
Suppose $\lambda, \mu \in \mathbb{F}_p$ such that $1+\lambda+\mu \neq 0$.  Then,
\begin{multline*}
 \Ln{2}{\frac{1}{3}}{\frac{1}{3} &\frac{1}{3}}{1}{1-\lambda^3, 1-\mu^3}_{\frac{p-1}{3}}\\
\equiv 
 \Ln{2}{\frac{1}{3}}{\frac{1}{3} &\frac{1}{3}}{1}{\left(\frac{1+\omega \lambda + \omega^2 \mu}{1+\lambda+\mu} \right)^3, \left(\frac{1+\omega^2 \lambda + \omega \mu}{1 + \lambda + \mu}\right)^3}_{\frac{p-1}{3}}
\pmod p.
\end{multline*}
\end{lemma}

\begin{proof}

Given a prime $p\equiv 1 \pmod 3$, setting $c=1-p$ in Theorem \ref{thm:MO} gives 
\begin{multline*}
\left( \frac{1+\lam+\mu}{3} \right)^{1-p} 
\Ln{2}{\frac{1-p}{3}}{\frac{1}{3}-\frac{p}{6} &\frac{1}{3}-\frac{p}{6}}{1-\frac{p}{2}}{1-\lambda^3, 1-\mu^3}\\
=
\Ln{2}{\frac{1-p}{3}}{\frac{1}{3}-\frac{p}{6} &\frac{1}{3}-\frac{p}{6}}{1-\frac{p}{6}}{\left(\frac{1+\omega \lambda + \omega^2 \mu}{1+\lambda+\mu} \right)^3, \left(\frac{1+\omega^2 \lambda + \omega \mu}{1 + \lambda + \mu}\right)^3}
\end{multline*}
Observe that since $\frac{1-p}{3}$ is a negative integer, the hypergeometric function on each side of the above equality truncates naturally at the power $\frac{p-1}{3}$.  Also, since $1 + \lambda + \mu \neq 0$, then by Fermat's Little Theorem, we have $\left( \frac{1+\lam+\mu}{3} \right)^{1-p} \equiv 1 \pmod{p}$.  The left-hand side is therefore congruent modulo $p$ to 
$F^{(n)}_D\left[\begin{matrix}\frac{1}{3};& \frac{1}{3} & \frac{1}{3}  \smallskip \\  {} & 1 \end{matrix} \; ; \; 1-\lam^3, 1-\mu^3 \right]_{\frac{p-1}{3}}$ .
Similarly the right-hand side is congruent to 
$F^{(n)}_D\left[\begin{matrix}\frac{1}{3};& \frac{1}{3} & \frac{1}{3}  \smallskip \\  {} & 1 \end{matrix} \; ; \; 
\left(\frac{1+\omega \lam + \omega^2 \mu}{1+\lam+\mu} \right)^3, \left(\frac{1+\omega^2 \lam + \omega \mu}{1+\lam+\mu} \right)^3 \right]_{\frac{p-1}{3}}$ .
Combining the congruences gives the desired result.
\end{proof}

We next give a proof of  Corollary \ref{finite2F1-cubic}, the finite-field analogue of Borwein and Borwein's cubic arithmetic-geometric mean, which follows from the theorem above.  We restate the corollary here for convenience.

\newtheorem*{finite2F1-cubic}{Corollary \ref{finite2F1-cubic}}	
\begin{finite2F1-cubic}
For $p\equiv 1 \pmod{3}$ prime, and $\omega$ as above, if $\lam \in \FF_p$ satisfies $1 + 2\lam \neq 0$,  then
\begin{align*}
\pFFq{2}{1}{\eta_3 & \eta_3^2}{&\eps}{1-\lam^3} 
=
\pFFq{2}{1}{\eta_3 & \eta_3^2}{&\eps}{\left(\frac{1- \lam}{1 + 2\lam}\right)^3}.
\end{align*}
\end{finite2F1-cubic}

\begin{proof}
Let $\lambda \in \FF_p$ with $1 + 2\lam \neq 0$, and put $\zeta := \frac{1+\omega \lambda + \omega^2 \lam}{1 + 2 \lambda} = \frac{1-\lam}{1+2\lam}$.  The corollary holds trivially when $\lam = 0$ or $1$, so suppose that $\lam \neq 0, 1$, and consequently $\zeta \neq 0$.  Applying Proposition \ref{FF-schlosser33a} to the left-hand of Theorem \ref{finiteF1-cubic} gives
\begin{align*}
\begin{split}
\FAFn{2}{\eta_3}{\eta_3&  \eta_3}{\eps}{1-\lam^3, 1-\lam^3 } &= \eps(1-(1-\lambda^3))  \, \pFFq21{\eta_3^2 &  \eta_3}{&\eps}{1-\lam^3}
+\delta(1-(1-\lambda^3)) \frac{J(\eta_3, \eps)}{J(\eta_3,\eta_3^2)} \\
&= \eps(\lambda^3)  \, \pFFq21{\eta_3^2 &  \eta_3}{&\eps}{1-\lam^3}
+\delta(\lambda^3) \frac{J(\eta_3, \eps)}{J(\eta_3,\eta_3^2)} \\
& = \pFFq21{\eta_3^2 &  \eta_3}{&\eps}{1-\lam^3}.\\
\end{split}
\end{align*}
Similarly, since $\zeta \neq 0$, applying the proposition to the right-hand side of Theorem \ref{finiteF1-cubic} gives
\begin{align*}
\begin{split}
\FAFn{2}{\eta_3}{\eta_3 &  \eta_3}{\eps}{\zeta^3, \zeta^3}
= \eps(1-\zeta^3)  \, \pFFq21{\eta_3^2 &  \eta_3}{&\eps}{\zeta^3}
+\delta(1-\zeta^3) \frac{J(\eta_3, \eps)}{J(\eta_3,\eta_3^2)} = \pFFq21{\eta_3^2 &  \eta_3}{&\eps}{\zeta^3}.
\end{split}
\end{align*}
The result then follows from the symmetry of the $\FFtwo$-function in the characters $A$ and $B$.
\end{proof}


\section{Acknowledgements}\label{section:thanks}
We thank Ling Long for providing the inspiration for this work.  We also thank the organizers of the ICERM special semester program on Computational Aspects of the Langlands Program, and the organizers of the MATRIX workshop on Hypergeometric Motives and Calabi-Yau Differential Equations, as these programs fostered this collaboration.  
\bigbreak





\begin{thebibliography}{99}

\bibitem{ahlono1} S.~Ahlgren and K.~Ono, \emph{A {G}aussian hypergeometric series evaluation and {A}p\'ery number congruences}, J. Reine Angew. Math. \textbf{518} (2000), 187--212.

\bibitem{ahlono2} S.~Ahlgren and K.~Ono, \emph{Modularity of a certain {C}alabi-{Y}au threefold}, Monatsh. Math. \textbf{129} (2000), 177--190.

\bibitem{archinard} N.~Archinard, \emph{Hypergeometric abelian varieties}, Canad. J. Math. Vol. 55 (5), 2003,  897--932.

\bibitem{AAR} G.~Andrews, R.~ Askey, and R.~Roy. {\em Special Functions} Cambridge University Press, Cambridge, 1999.

\bibitem{appell} P.~Appell, \emph{Sur les fonctions hyperg\'{e}om\'{e}triques de deux variables}, Journal de Math\'{e}matiques Pures et Appliq\'{u}ees 3e s\'{e}rie, 8 (1882) 173--216.

\bibitem{appell2} P.~Appell, Sur les Fonctions hyp\'{e}rg\'{e}ometriques de plusieurs variables les polynomes d'Hermite et autres fonctions sph\'{e}riques dans l'hyperspace, Gauthier--Villars, Paris, 1925.

\bibitem{appell-kampe} P.~Appell, J.~ Kamp\'{e} de F\'{e}riet, Fonctions hyperg\'{e}om\'{e}triques et hypersph\'{e}riques; Polyn\^{o}mes d'Hermite. Gauthier-Villars, Paris (1926).

\bibitem{beukers-etal} F.~Beukers, H.~Cohen, A.~Mellitt, \emph{Finite hypergeometric functions}, arXiv:1505.02900.

\bibitem{borweins-1} J.~M.~Borwein, P.~M.~Borwein, \emph{A Remarkable Cubic Mean Iteration}, Proceedings of the Valparaiso conference, St. Ruseheweyh, E. B. Saff, L. C. Salinas, R.S. Varga (eds.), Springer Lecture Notes in
Mathematics, 1435 (1989) 27--31.

\bibitem{borweins-2} J.~M.~Borwein, P.~B.~Borwein, \emph{A cubic counterpart of Jacobi's identity and the AGM}, Trans. Amer. Math. Soc. 323 (2) (1991) 691--701.



\bibitem{WIN3a}
 A.~Deines, J.~Fuselier, L.~Long, H.~Swisher, and F.~Tu. \emph{Generalized Legendre Curves and Quaternionic Multiplication.}  Journal of Number Theory (in press), arXiv:1412.6906, 2014.

\bibitem{WIN3b}  
A.~Deines, J.~Fuselier, L.~Long, H.~Swisher, and F.~Tu. \emph{Hypergeometric series, truncated hypergeometric series, and Gaussian hypergeometric functions.} ArXiv e-prints 1501.03564, Proceedings for Women in Number Theory 3 workshop, Association for Women in Mathematics series, January 2015.

\bibitem{FOP} S.~Frechette, K.~Ono, and M.~Papanikolas, \emph{Gaussian Hypergeometric Functions and Traces of Hecke
Operators,}  Internat. Math. Res. Notices, {\bf 60} (2004), 3233--3262.

Natl. Acad. Sci. USA, to appear.



\bibitem{dictionary} J.~Fuselier, L.~Long, R.~Ramakirshma, H.~Swisher, F.~Tu, \emph{Hypergeometric functions over finite fields}, arXiv:1510.02575v2

\bibitem{greene} J.~Greene, \emph{Hypergeometric functions over finite fields}, Trans. Amer. Math. Soc., {\bf 301} (1987), 77--101.

\bibitem{he} B.~He, A Lauricella hypergeometric function over finite fields, arXiv:1610.04473.


\bibitem{ireland-rosen} K.~Ireland \& M.~Rosen.  {\em A classical introduction to modern number theory, volume 84 of Graduate
Texts in Mathematics.}  Springer-Verlag, New York, second edition, 1990.



\bibitem{koike-shiga1} K.~Koike, H.~Shiga, \emph{Isogeny formulas for the Picard modular form and a three terms arithmetic geometric mean}, J. of Number Theory, 124 (2007), 123--141.

\bibitem{koike-shiga2} K.~Koike, H.~Shiga, \emph{An extended Gauss AGM and corresponding Picard modular forms}, J. of Number Theory, 128 (2008), 2097--2126.


\bibitem{lauricella} G.~Lauricella,  {\em Sulle funzioni ipergeometriche a pi\`{u} variabili,} Rendiconti del Circolo Matematico di Palermo (in Italian),  7 (S1), (1893), 111--158.

\bibitem{li} L.~Li, X.~Li, R.~Mao, Some new formulas for Appell series over finite fields, arXiv:1701.02674.


\bibitem{matsumoto-ohara}K.~Matsumoto, K.~Ohara, \emph{Some transformation formulas for Lauricella's hypergeometric functions $F_D$}, (English summary)  Funkcial. Ekvac. 52 (2009), no. 2, 203--212.

\bibitem{mccarthy} D.~McCarthy, Extending Gaussian hypergeometric series to the $p$-adic setting, {\em Interna-
tional J. of Number Theory}, 8(7), 1581--1612, 2012.

\bibitem{mccarthy2} D.~McCarthy, {\emph On a supercongruence conjecture of Rodriguez-Villegas},  Proceedings of the American Mathematical Society, 140(7), (2012) 2241--2254.


\bibitem{picard} E.~Picard, \emph{Sur une extension aux fonctions de deux variables du probl\`{e}me de Riemann relativ aux fonctions hyperg\'{e}om\'{e}triques},  Annales scientifiques de l'\'{E}cole Normale Sup\'{e}rieure. (2\`{e}me s\'{e}rie) 10 (in French), (1881), 305--322. 

\bibitem{picard2} E.~Picard, \emph{Sur les fonctions de deux variables independentes analogues aux fonctions modulaires,} Acta
Math., 2 (1883), 114--135.


\bibitem{schlosser} M.~Schlosser, \emph{Multiple hypergeometric series -- Appell series and beyond}, Computer Algebra in Quantum Field Theory, C.~Schneider, J.~Bl\"{u}mlein (eds.), Springer Texts \& Monographs in Symbolic Computation, (2013) 305--345.




\bibitem{slater} Lucy Joan Slater, {\em Generalized hypergeometric functions.} Cambridge University Press, Cambridge, 1966.

\bibitem{vega} M. Valentina Vega. {\em Hypergeometric functions over finite fields and their relations to algebraic curves,}  Int. J. Number Theory, 7(8), (2011), 2171--2195.

\end{thebibliography}
\end{document}